\documentclass[draft, 12pt, reqno]{amsart} 
\usepackage{amssymb,latexsym,amsfonts,verbatim,amscd,ifthen} 
\usepackage{color}
\usepackage{relsize}
\usepackage[utf8]{inputenc}
\usepackage{mathrsfs}
\usepackage{mathbbol}

\newcommand{\C}{{\mathbb C}}
\newcommand{\R}{{\mathbb R}}

\newcommand{\Pp}{{\mathbb P}}
\newcommand{\h}{{\mathcal H}}
\newcommand{\lb}{{\lambda}}
\numberwithin{equation}{section}

\theoremstyle{plain}

\newtheorem{theorem}{Theorem}[section]
\newtheorem{lemma}[theorem]{Lemma}
\newtheorem{proposition}[theorem]{Proposition}
\newtheorem{corollary}[theorem]{Corollary}
\newtheorem{problem}[theorem]{Problem}
\newtheorem{example}[theorem]{Example}

\newtheorem*{definition*}{Definition}

\newtheorem{remark}[equation]{Remark} 
\newtheorem{definition}[equation]{Definition}

\begin{document}
\title[Projective spectra and common invariant subspaces]{Spectral test of reducibility for matrix tuples}
\author{Michael  Stessin}
\address{Department of Mathematics and Statistics \\
University at Albany \\
Albany, NY 12222}
\email{mstessin@albany.edu}
\author{Rongwei Yang} 
\address{Department of Mathematics and Statistics \\
University at Albany \\
Albany, NY 12222}
\email{ryang@albany.edu}
\begin{abstract}
If a tuple of matrices has a common invariant subspace, its projective joint spectrum has an algebraic component.  In general, the converse is not true, and there might be algebraic components in the projective joint spectrum without corresponding common invariant subspaces. In this paper we give necessary and sufficient conditions for the occurrence  of such correspondence.

\end{abstract}
\maketitle

\section{\textbf{Introduction}}


Problems concerning common invariant subspaces naturally appear in the study of representations $\rho$ of a group or an algebra ${\mathcal A}$. There are two related but different questions: 1) Is $\rho$ reducible? 2) Is $\rho$ decomposable into a direct sum of two sub-representations? The answer to 1) is positive if and only if the linear operators $\rho(g),\ g\in {\mathcal A}$, have a nontrivial common invariant subspace; while 2) is positive if and only if the operators have a nontrivial common reducing subspace. In the case ${\mathcal A}$ is generated by a finite set $\{a_1, ..., a_n\}$, the two questions lead naturally to the discussion on the reducibility of the projective joint spectrum $\sigma(A)$, where $A_i=\rho(a_i), i=1, ..., n$. If matrices $A_1,...,.A_n$ have a common invariant subspace, then it is not hard to see that $\sigma(A)$ has an algebraic component. The converse is not true: there are simple examples of matrices, whose projective joint spectra have algebraic components, that have no common invariant subspaces. The issue of finding necessary and sufficient conditions for the existence of common reducing subspaces was considered in \cite{S}. However, the conditions found there are generally not applicable to the description of common invariant (but not reducing) subspaces.


The goal of this paper is to establish some necessary and sufficient conditions for the correspondence between proper components of $\sigma(A)$ and common ivariant subspaces.
Our approach employs the techniques developed in recent works \cite{S,Y1}, and it demonstrates a special role played by projections belonging to the unital $C^*$ algebra generated by the matrices of the tuple and their adjoints. 

\vspace{.2cm}

Given a $n$-tuple of $N\times N$ matrices $A=(A_1,...,A_n)$, the linear combination $A(x):=x_1A_1+\cdots +x_nA_n, x_j\in \C$ is often called a {\em linear pencil} of $A$.
Its determinant $\det A(x)$ is a homogeneous polynomial of degree $N$ in the variables $x:=(x_1,...,x_n)\in \C^n$. Zeros of this polynomial constitute an algebraic variety in the projective space $\C\Pp^{n-1}$, called the \textit{determinantal variety} (or \textit{determinantal hypersurface}) of the tuple $A$. We use the following notation \small$$\sigma(A)=\big\{x=[x_1:\cdots :x_n]\in \C\Pp^{n-1}: \ \det A(x)=0\big\}.$$\normalsize
An infinite dimensional analog of determinantal variety, called \textit{projective joint spectrum} of a tuple of operators acting on a Hilbert space $\h$, was introduced in \cite{Y}. It is defined by \[\sigma(A)=\big\{x\in \C\Pp^{n-1}: \ A(x) \ \mbox{is not invertible}\big\}.\]

To avoid trivial redundancies it is frequently assumed that at least one of the operators is invertible, and, thus, can be assumed to be the identity (or rather $-I$). In what follows we always assume that $A_{n+1}=-I$ and write $\sigma(A)$ instead of $\sigma(A,-I)$.

Determinantal varieties of matrix tuples have been under scrutiny for more than a hundred years. Notably, the study of group determinants led Frobenius to laying out the foundation of representation theory. There is an extensive literature on the question when an algebraic variety of codimension 1 in a projective space admits a determinantal representation. Without trying to give an exhausting account of the references on this topic, we just mention \cite{D1}-\cite{D5}, \cite{KV, V}, as well as the monograph \cite{D} and references therein.

In the last 10-15 years projective joint spectra of operator and matrix tuples have been intensely investigated, for instance see  \cite{BCY, CY, CSZ}, \cite{DY} - \cite{GLX}, \cite{MQW} - \cite{PS2}, \cite{S1} - \cite{Y1}). An underlying question is as follows.
\begin{problem}
What does the geometry of $\sigma(A)$ tell us about the relations among $A_1, ..., A_n$? 
\end{problem}
Relating to group theory, \cite{CST, GY, PS1} characterize representations of infinite dihedral group, non-special finite Coxeter groups, and some classical finite groups related to the Hadamard matrices of Fourier type in terms of projective joint spectra. Parallelly, spectral properties of Lie algebra representations were investigated in \cite{AY, CCD, GLX, HZ}. Furthermore, spectral rigidity for the infinitesimal generators of representations of twisted $S_\nu U(2)$ was established in \cite{S2}. A recent monograph \cite{Y1} is a good source for information on the subject. A study that underpins the investigation here is \cite{S3}, which proves that for an arbitrary operator $T$ acting on $\h$ the projective joint spectrum $\sigma (D, S, S^*, T)$ determines $T$, where $S$ is the unilateral shift, and $D$ is a fixed bounded diagonal operator, $D=diag(\lambda_1,\lambda_2,...)$, whose diagonal entrees $\lambda_j$ are distinct, and none of them is an accumulation point of the set $\Lambda=\{\lambda_j: \ j=1,2,...\}$.

For a pair $A=(A_1, A_2)$ of bounded self-adjoint operators on $\h$, the question when the appearance of an algebraic component in $\sigma(A)$ implies the existence of a common invariant subspace for $A_1$ and $A_2$ was initially considered in \cite{ST}. A sufficient condition presented there is applicable to a restricted class of operators consisting of compact operators plus scalar multiples of the identity. It was expressed in terms of the projective joint spectra of tensor powers of operators acting on the corresponding exterior power of $\h$. This condition was by far not necessary. Recently, paper \cite{S} gives some necessary and sufficient conditions relating algebraic components of $\sigma(A)$ with finite multiplicity to common reducing subspaces. In general, if a tuple is non-self-adjoint, a common invariant subspace is not necessarily reducing. The goal of this paper is to present some necessary and sufficient conditions on this issue for an arbitrary matrix tuple. 



The structure of this paper is as follows. In section \ref{admis} we establish an existence of what we call ``admissible transformations". Those are linear maps that transform our tuple into a one whose projective joint spectrum has no singular points on coordinate projective lines. Passing to an admissible tuple allows us to associate components of the projective joint spectrum with projections on invariant subspaces, and these projections are elements of the algebra generated by the matrices in the tuple. This is done in section \ref{pr}. Finally, in section \ref{m} we prove our main results that give necessary and sufficient conditions for a proper component of the projective joint spectrum to correspond to a common invariant subspace.

\section{Admissible Transformations}\label{admis}

First, we  remark  that, since adding to an operator a multiple of the identity does not change the lattice of invariant subspaces,  when dealing with common invariant subspaces we may assume without loss of generality that all operators $A_j$ are invertible. In the sequel, we consider $N\times N$ matrices $A_1,...,A_n$ and the pencil $A(x)=x_1A_1+\cdots +x_nA_n-x_{n+1}I$. Suppose that
\begin{eqnarray}
\det A(x)=R_1(x)^{m_1} R_2(x)^{m_2}\cdots R_k(x)^{m_k}, \label{polynomial}
\end{eqnarray}
where each $R_s$ is an irreducible homogeneous polynomials of degree $l_s$. Thus, we have $l_1m_1+\cdots +l_km_k=N$. The invertibility of $A_1,...,A_n$ implies that for each $1\leq s\leq k$ and $1\leq j\leq n$ the intersection of the variety $\{R_s=0\}\subset \C\Pp^{n}$ with the coordinate projective line 
$$L_j=\big\{ [0:\cdots :x_j: 0: \cdots :1]: \ x_j\in \C \big\}$$ consists of $l_s$ points counting multiplicity.  
Following \cite{ST}, we call the part of $\sigma(A)$ that lies in the chart $\{x_{n+1}=1\}\subset \C\Pp^{n}$ the \textit{proper projective spectrum} of the tuple $A$ and denote it by $\sigma_p(A)$. For simplicity, in the sequel we shall identify the chart $\{x_{n+1}=1\}\subset \C\Pp^{n}$ with $\C^n$. Thus $\sigma_p(A)$ is a subset of $\C^n$.
Evidently, $\sigma(A)\cap L_j \subset \sigma_p(A), \ j=1,...,n. $

\vspace{.2cm}

For an $n\times n$ complex matrix $C=[c_{ij}]_{i,j=1}^n$, consider the following tuple transformation:
\begin{equation}\label{transform}
\widehat{A}_j=\displaystyle \sum_{s=1}^n c_{js}A_s, \ j=1,...,n. 
\end{equation} 

It is easy to see that the projective joint spectrum of $\widehat{A}_1,...,\widehat{A}_n$ is given by
$$\sigma(\widehat{A})=\Big\{\Big( \displaystyle \prod_{s=1}^k R_s^{m_s}(x\mathscr{C})\Big)=0\Big\}, $$
where $\mathscr{C}$ is the $(n+1)\times (n+1)$ obtained from $C$ by adding the $(n+1)$-th  column $(0,...,0,1)^T$ and the $(n+1)$-th row $(0,...,0,1)$.

\begin{definition}\label{admissible}
We call a transformation \eqref{transform} \textit{admissible}, if
\begin{itemize}
\item[(1)] the matrix $C$ is invertible and, if the tuple $A$ consists of self-adjoint matrices, $C$ is real-valued;
\item[(2)] for each $1\leq j\leq n$, at every point of intersection of $\sigma(\widehat{A}) $ with $L_j$, the derivative of $ \prod_{s=1}^k R_s(x\mathscr{C})$ with respect to $x_j$ is not equal to 0.
\end{itemize}	
\end{definition}
\noindent Part (2) above indicates that every point of intersection of $\sigma(\widehat{A})$ with $L_j$ is a regular point of the algebraic variety $\{ \prod_{s=1}^k R_s(x\mathscr{C})=0\}$. In the sequel, tuples with this spectral property will be called {\em admissible tuples}. Since every regular point has multiplicity 1, we see that the intersection of $\big\{R_s(x)=0\}$ with $L_j$ consists of $l_s$ distinct points, and these sets of points are different for different $s$. The following Theorem extends \cite[Theorem 1.11]{S}. 

\begin{theorem}\label{admissible tr}
There are admissible transformations in every neighborhood of the identity.	
\end{theorem}

\begin{proof}
In fact, we will show that the set of admissible transformations is open and dense in a neighborhood of the identity. To this end, it suffices to prove that for each $1\leq j \leq n$ and $1\leq s \leq  k$ the set of invertible matrices $\mathscr{C}$ for which the intersection of
\begin{equation}\label{hat r}
{\widehat{\Gamma}}_s := \Big\{x\in \C\Pp^{n}:\ R_s(x\mathscr{C})=0\Big\}
\end{equation}
with the coordinate line $L_j$ consists solely of regular points satisfying condition (2) of definition \ref{admissible} is open and dense in a neighborhood of the identity matrix. As  mentioned above, the coordinate line $L_j$ meets $\Gamma_s=\big\{ R_s(x)=0\big\}$ at $l_s$ points counting multiplicity, all of them belonging to the chart $\big\{x_{n+1}=1\big\}=\C^n$. Let us denote by ${\mathcal O}_{1,j},...,{\mathcal O}_{l_s,j}$ some disjoint neighborhoods of these points that separate each one from the others. 
We set \[\widetilde{R}_s(x_1,...,x_n):=R_s(x_1,...,.x_n,1), \ \ \ x\in \C^n.\] Then $\Gamma_s$ can be identified as the zero set of $\widetilde{R}_s$. Observe that since $R_s$ is irreducible and homogeneous of degree $l_s$, by factoring $x_{n+1}^{l_s}$ out of $R_s$, we see that $\widetilde{R}_s$ is also irreducible. The directional derivative (or {\em Euler field} \cite{Gri}) in $\C^n$ is defined as $\theta=\sum_{j=1}^{n}x_j\frac{\partial}{\partial x_j}$
For $z=(z_1,...,z_n)\neq 0$, let $\ell_z$ denote the line in $\C^n $ that passes through the origin and $z$, i.e., $$\ell_z=\big\{ (\lambda z_1,...,\lambda z_n): \ \lambda \in \C \big\}. $$ 
 Suppose that $x\in \widetilde{\Gamma}_s$. Then the derivative of $\widetilde{R}_s$ in the direction of $\ell_x$ evaluated at $ y=(\lambda x_1,...,\lambda x_n)$ is 


\begin{eqnarray}
&\frac{\partial \widetilde{R}_s}{\partial \ell_x} \Big|_y :=x_1\frac{\partial  \widetilde{R}_s}{\partial x_1}\Big|_y+\cdots +x_n\frac{\partial  \widetilde{R}_s}{\partial x_n}\Big|_y=\frac{1}{\lambda}\theta (\widetilde{R}_s)(y). \label{directional derivative}
\end{eqnarray}

Consider the algebraic set 
\[ \Delta_s:=\{ \zeta \in \C^{n}:\ \widetilde{R}_s(\zeta)= \theta (\widetilde{R}_s)(\zeta)=0\}.\]
We claim that the dimension (over $\C$) of $\Delta_s$ is less than $n-1$. Indeed,  this dimension is obviously less than or equal to $n-1$, as $\Delta_s \subset {\Gamma}_s$ and $\dim {\Gamma}_s=n-1$. If $\dim \Delta_s=n-1$, then, since ${\Gamma}_s$ is irreducible, we must have $\Delta_s={\Gamma}_s$. Since the polynomial $\widetilde{R}_s$ is irreducible, the polynomial $ \theta (\widetilde{R}_s)$ must be a multiple of $\widetilde{R}_s$. If follows that $\theta (\widetilde{R}_s)$ must be a scalar multiple of $\widetilde{R}_s$, since the two polynomials have the same degree. Obviously this scalar is not 0, so the zero set of $ \theta (\widetilde{R}_s) $ is ${\Gamma}_s$, which is a contradiction, as $0$ is in the zero set of $ \theta (\widetilde{R}_s)(\zeta)$, but not in ${\Gamma}_s$.
 
Evidently, the singular locus of ${\Gamma}_s$ is contained in $\Delta_s$. Hence every point in ${\Gamma}_s \setminus \Delta_s$ is a regular point. 
Note that the quotient map \[{\mathscr F}: \C^{n} \setminus \{0\} \to \C\Pp^{n-1},\ \ \ {\mathscr F}(x)=\ell_x\]
restricted to ${\mathcal O}_{r,j}$ is differentiable. We claim that the image $ {\mathscr F}\big(({\Gamma}_s\setminus \Delta_s)\cap {\mathcal O}_{r,j}\big)$ is open and dense in ${\mathscr F}({\mathcal O}_{r,j})$, which is a neighborhood of $[\underbrace{0: \cdots : 0}_{j-1}:1:0: \cdots :0]$ in $\C\Pp^{n-1}$. Indeed, since $\dim(\Delta_s)\leq n-2$, Sard's theorem \cite{Sa} and the existence of Witney's stratification \cite{LT} imply that the measure of ${\mathscr F}(\Delta_s)$ is equal to 0, and ${\Gamma}_s \setminus \Delta_s$ is open and dense in ${\Gamma}_s$. For every $x\in {\Gamma}_s \setminus \Delta_s$, the line $\ell_x$ is not tangent to ${\Gamma}_s$, as otherwise we would have $x\in \Delta_s$ by \eqref{directional derivative}. This implies that there exists a small neighborhood $V$ of $x$ such that the restriction of ${\mathscr F}$ to ${\Gamma}_s\cap V$ is an isomorphism, implying that ${\mathscr F}({\Gamma}_s\cap V)$ is an open neighborhood of ${\mathscr F}(x)$ in $\C\Pp^{n-1}$ and consequently has a positive measure. Finally, since the image of a compact set under ${\mathscr F}$ is closed, we see that, if ${\mathcal O}_{r,j}$ sufficiently small, $\Delta_s\cap {\mathscr F}({\mathcal O}_{r,j})$ is closed, which finishes the verification of our claim.

As a result, we see that the set of lines $\ell_x, \ x\in {\Gamma}_s$ close to the $j$-th coordinate line $L_j$ that do not meet $\Delta_s$ is open and dense in a small neighborhood of $L_j$. Equation \eqref{directional derivative} shows that every point of the intersection of such line with ${\Gamma}_s$ is a regular point. 
To proceed, we set $$ \Xi_{rs}={\Gamma}_r \cap {\Gamma}_s= \big\{x\in \C^{n}: \ \widetilde{R}_s(x)=\widetilde{R}_r(x)=0 \big\}, \ r,s=1,...,k.$$
As $\widetilde{R}_s$ and $\widetilde{R}_r$ are irreducible, $\Xi_{rs}$ is an algebraic set of dimension less than or equal to $n-2$. An argument similar to the one above shows that ${\mathscr F}(\Xi_{rs})$ is closed and has measure 0. Thus $$ \omega_{s,j}:={\Gamma}_s \setminus \big (\Delta_s \cup_{r\neq s}{\mathscr F}(\Xi_{rs})\big)$$
is open and dense in $\cap_{r=1}^{l_s}{\mathscr F}\big({\mathcal O}_{r,j}\big)$.

Finally, we set $\Omega_j=\cap_{s=1}^k \omega_{s,j}. $ Then for every $x\in \Omega_j$ the line $\ell_x$ intersects the component $\{R_s=0\}$ of $\sigma(A_1,...,A_n)$ at $l_s$ distinct points
each of which is a regular point and has multiplicity $m_s$ in $\sigma(A_1,...,A_n)$.

For every $j=1,...,n$, choose a point $\xi_j=(\xi_{1}^j,...,\xi_{n}^j)\in \sigma(A_1,...,A_n)$ such that $\ell_{\xi_j} \in \Omega_j$ and consider the following matrix:
\begin{equation} \label{admiss}
C_{\Xi }=\left [ \begin{array}{ccc} \xi_1^1 & \cdots & \xi_n^1 \\ \cdot & \cdot & \cdot \\ \xi_1^n & \cdots & \xi_n^n \end{array} \right ]. 
\end{equation} 
If all neighborhoods ${\mathcal O}_{rj}$ are small enough, then $\xi_j$ is close to $L_j$, and hence the matrix $C_\Xi$ is invertible. Furthermore,  for each $1\leq s\leq k$ and every point $y\in L_j$,  it follows from \eqref{directional derivative} that $$\frac{\partial {R}_s(x{\mathscr C}_\Xi)}{\partial x_j}\Big|_y=\frac{\partial \widetilde{R}_s}{\partial \ell_{\xi_j}}\Big|_y\neq 0,$$and this completes the proof.
\end{proof}

\vspace{.5cm}
\begin{remark}\label{2.9}
Combining the above theorem with Theorem 1.11 and Lemma 4.1 in \cite{S}, we see that if all matrices $A_1,...,.A_n$ are self-adjoint, then in every neighborhood of the identity matrix there is a real-valued admissible transformation.	
\end{remark}

\section{Construction of Projections}\label{pr}

In the sequel, a projection refers to a linear operator $P$ such that $P^2=P$. Such an operator is also called an idempotent in the literature. 
Since a subspace $M\subset \C^N$ is invariant for $A_1, ..., A_n$ if and only if it is invariant for $\widehat{A}_1, ..., \widehat{A}_n$, in light of Theorem \ref{admissible tr} we assume without loss of generality that the tuple $A=(A_1, ..., A_n)$ is admissible throughout the rest of the paper. For a fixed $j$ and each $1\leq s\leq k$, we write \[ \Gamma_s\cap L_j=\{t_{j,s,1},...,t_{j,s,l_s}\}.\] Observe that each $1/t_{j,s,r}, \ s=1,...,k, \ r=1,...,l_s$ is an eigenvalue of $A_j$ of multiplicity $m_s$. For a subset $S \subset \{1,...,k\}$, we let $S^c$ denote the complement $\{1,....,k\}\setminus S$. We will now express the projections onto the subspaces associated with $S$ as elements of the free algebra generated by ${A}_1,..., {A}_n$. The method is based on functional calculus and Cayley-Hamilton theorem.

First, consider a $N\times N$ matrix $T$ with distinct eigenvalues $\lb_1, ..., \lb_k$ and corresponding multiplicities $m_1, ..., m_k$. Then the characteristic polynomial of $T$ is \[q(z)=\prod_{s=1}^k (z-\lb_s)^{m_s}.\] With respect to any subset $S \subset \{1,...,k\}$, we consider the decomposition $q(z)=q_S(z)q_{S^c}(z)$, where $q_S(z)=\prod_{s\in S} (z-\lb_s)^{m_s}$. Evidently, $q_S(\lb_s)\neq 0$ for each $s\in S^c$. Define 
\begin{equation}\label{idem}
\alpha_S=(-1)^{|S^c|}\prod_{s\in S^c} \left(q_S(\lb_s)\right)^{m_s},\ \ \  p_S(z)=\prod_{s\in S^c} (q_S(z)-q_S(\lb_s))^{m_s}.
\end{equation}
Observe that $p_S(\lb_s)=\alpha_S$ if $s\in S$ and $p_S(\lb_s)=0$ if $s\in S^c$. We set ${\mathscr P}_{S}=p_S(T)/\alpha_S$. Then spectral mapping theorem implies that $\sigma({\mathscr P}_{S})=p_S(\sigma (T))/\alpha_S=\{0, 1\}$. In fact, more is true. The following lemma helps illustrate the subsequent constructions. 
\begin{lemma}\label{idem2}
 Given any subset $S \subset \{1,...,k\}$, the matrix ${\mathscr P}_{S}$ is a projection.
\end{lemma}
\begin{proof}
First, we assume $T$ is in its Jordan normal form. Then, direct computation based on the Cayley-Hamilton theorem verifies that the polynomial $p_S/\alpha_S$ annihilates the Jordan cells associated with the eigenvalues $\lb_s, s\in S^c$, and for each $s\in S$ its evaluation at the Jordan cell associated with $\lb_s$ is the identity matrix on the invariant subspace corresponding to $\lb_s$. It follows that ${\mathscr P}_{S}$ is a diagonal projection matrix. Since, every matrix is similar to its Jordan normal form, the lemma follows.
\end{proof}

Note that if $T$ is self-adjoint, then we can omit the power $m_s$ in the definition (\ref{idem}). Thus, to continue with the discussion on the matrices $A_1, ..., A_n$, we consider two cases. 

{\bf (a)}. {\em Tuple $A$ consists all of self-adjoint matrices.} Define polynomials
\begin{eqnarray}\label{polynomial q self adjoint}
q_{j, S}(z)=\displaystyle \prod_{s\in S}\prod_{r=1}^{l_s}\Big(z-\frac{1}{t_{j,s,r}}\Big),\ \ \ j=1, ..., n.
\end{eqnarray}
Observe that each $1/t_{j,s,r}$ is an eigenvalue of $A_j$, and hence the function $q_{j,S}$ is a factor of the characteristic polynomial of $A_j$. Moreover, we have
$$q_{j,S}\Big( \frac{1}{t_{j,s,r}}\Big)=0, \ s\in S, \ r=1,..., l_s;\ \ q_{j,S}\Big( \frac{1}{t_{j,s,r}}\Big)\neq 0 \ \mbox{for} \ s\in S^c.$$






Now, we write
$$\alpha_{j,S}= \displaystyle \prod_{s\in S^c} (-1)^{l_s}\prod_{r=1}^{l_s}q_{j,S}\big(\frac{1}{t_{j,s,r}}\big)$$
and define
\begin{equation}\label{projection S self adjoint}
{\mathscr P}_{j,{S}}=\frac{1}{\alpha_{j, S}}\displaystyle \prod_{s\in S^c}\prod_{r=1}^{l_s}\Big(q_{j,S}(A_j)-q_{j,S}( \frac{1}{t_{j,s,r}}\big)\Big).
\end{equation}

{\bf (b)}. \textit{Tuple $A$ contains at least one non-self-adjoint matrix}. In this case, the definition of polynomials \eqref{polynomial q self adjoint} - \eqref{projection S self adjoint} is slightly different as matrices $A_1,...,A_n$ might not be diagonalizable. Set
\begin{eqnarray}
& \hat{q}_{j,{S}}(z)=\displaystyle \prod_{s\in{S}}\prod_{r=1}^{l_s}\Big(z-\frac{1}{t_{j,s,r}}\Big)^{m_s} \label{polynomial q} 
\end{eqnarray}
Then again
$$\hat{q}_{j,{S}}\Big( \frac{1}{t_{j,s,r}}\Big)=0, \ s\in {S}, \ r=1,..., l_s, \ \hat{q}_{j,S}\Big( \frac{1}{t_{j,s,r}}\Big)\neq 0 \ \mbox{for} \ s\in S^c.$$

We write
$$\hat{\alpha}_{j,S}= \displaystyle \prod_{s\in S^c} (-1)^{l_s}\prod_{r=1}^{l_s}\hat{q}_{j,S}\big(\frac{1}{t_{j,s,r}}\big)^{m_s}$$
and define
\begin{equation}\label{projection S}
\widehat{{\mathscr P}}_{j,S}=\frac{1}{\hat{\alpha}_{j,S}}\displaystyle \prod_{s\in S^c}\prod_{r=1}^{l_s}\Big(\hat{q}_{j,S}\big({A}_j)-\hat{q}_{j,S}\big(\frac{1}{t_{j,s,r}}\big)\Big)^{m_s}.
\end{equation}

When the set $S$ consists of a single number $s$ we will write ${\mathscr P}_{js}$ and $\widehat{{\mathscr P}}_{js}$. Since  $t_{j,s,r}$ are distinct complex numbers for different $(s,r)$, there are contours $\gamma_{j,s,r}$ that separates $\frac{1}{t_{j,s,r}}$. Then
\begin{equation}\label{projections for A j}
{\mathcal P}_{j,s,r}=\frac{1}{2\pi i}\int_{\gamma_{j,s,r}} (w-{A}_j)^{-1}dw	
\end{equation}
is the projection on the ${A}_j$- invariant subspace $L_{j,s,r}$ corresponding to the eigenvalue $\frac{1}{t_{j,s,r}}$. 

In light of Theorem \ref{admissible tr}, we assume without loss of generality that the identity matrix is admissible for the matrix $(A_1, ..., A_n)$. The following proposition is an immediate consequence of Lemma \ref{idem2} and its proof.

\begin{proposition}\label{script P}
For every ${S}\subset \{1,...,k\} $ and $1\leq j\leq n$, it holds that 
\begin{itemize}
\item[(1)] ${\mathscr P}_{j,{S}}$  and $\widehat{{\mathscr P}}_{j,{S}}$  are projections;
\item[(2)] ${\mathscr P}_{j,{S}}=\displaystyle \sum_{s\in {S}}\sum_{r=1}^{l_s} {\mathcal P}_{j,s,r}$, and the same is true for $\widehat{{\mathscr P}}_{j,{S}}$;
\item[(3)] the kernel of ${\mathscr P}_{j,S}$ \  ($\widehat{{\mathscr P}}_{j,S}$ ) is the range of ${\mathscr P}_{j,S^c}$ \ ($\widehat{{\mathscr P}}_{j, S^c}$):
$$ \ker ({\mathscr P}_{j,{S}})=\displaystyle \oplus_{s\in S^c}\oplus_{r=1}^{l_s} L_{j,s,r}.$$
\end{itemize}
\end{proposition}

\section{\textbf{Main Theorems}}\label{m}

If $L$ is a common invariant subspace for a tuple of $N\times N$ matrices $A_1,...,A_n$, then it is so for any set of non-commutative polynomials in $A_1,...,A_n$. In particular, it is so for any linear transform of type \eqref{transform}. Moreover, if the matrix $C$ of such transform is invertible, then the lattices of common invariant subspaces of the tuples $A=(A_1,...,A_n)$ and $\widehat{A}=(\widehat{A}_1,...,\widehat{A}_n)$  are the same. Therefore, as remarked before, Theorem \ref{admissible tr} allows us to assume without loss of generality that $A$ is an admissible tuple. This section proposes a criteria for determining when an appearance of a component of degree $k$ and multiplicity $m$ in $\sigma(A)$ implies the existence of a corresponding $km$-dimensional common invariant subspace. The study hinges on the projections ${\mathscr P}_{j,{\mathcal S}}$ and $\widehat{{\mathscr P}}_{j,{\mathcal S}}$ constructed in the preceding section. As before, there are two cases to consider. 

\subsection{Self-adjoint Tuples}

We first consider the case of self-adjoint operators. After possibly reordering the factors $R_j$ in the factorization (1.2), we may assume without loss of generality that $S=\{1, ..., r\}$, where $r\leq k$. Moreover, we set $q:=m_1+\cdots +m_r$ (see (1.2)).

\begin{theorem}\label{self adjoint}
Let $A_1,...,A_n$ be self-adjoint matrices. Then the union $\cup_{s\in {S} }\Gamma_s$ corresponds to a common invariant subspace if and only if 
\begin{eqnarray}\label{thm4.1}
&\sigma({\mathscr P}_{1,{S}},...,{\mathscr P}_{n,{S}})=\Big\{\big(x_1+\cdots +x_{n}-x_{n+1}\big)^{q}x_{n+1}^{N-q}=0\Big\}. \label{plane in the spectrum}
\end{eqnarray}	
\end{theorem}
\begin{proof}
First, by Proposition \ref{script P}, the projections ${\mathscr P}_{j,{S}}$ are orthogonal. The local spectral analysis (see \cite{S}, Theorem 2.6) at the point $(\underbrace{0,...,0}_{j-1}1,\underbrace{0,...,0}_{{n-j}})$ applied to the tuple $({\mathscr P}_{1,{S}},...,{\mathscr P}_{r,{ S}})$ shows 
$$ {\mathscr P}_{j,{S}} {\mathscr P}_{i,{S}} {\mathscr P}_{j,{ S}}= {\mathscr P}_{j,{S}},\ \ \ 1\leq i\leq n.$$
This means that the compression of $ {\mathscr P}_{i,{S}}$ onto the image of $ {\mathscr P}_{j,{S}}$ is the identity. Since the norm of $ {\mathscr P}_{j,{S}}$ is equal to 1, and ranks of these two  projections are the same, we see that $ {\mathscr P}_{i,{S}}= {\mathscr P}_{j,{S}}$ for all $1\leq i, j\leq n$, and equation (\ref{thm4.1}) follows. 

For the other direction, we recall that several normal matrices pairwise commute if and only if their projective joint spectrum is a union of hyperplanes \cite[Theorem 15]{CSZ}. Since ${\mathscr P}_{j,{S}}, j=1, ..., n$ are orthogonal projections, equation (\ref{thm4.1}) implies that they pairwise commute. Hence, without loss of generality, we can assume they are all diagonal. Furthermore, it is not hard to see that, for any square matrices $T_1, ..., T_n$ of equal size, if the hyperplane $\{\lb_1x_1+\cdots +\lb_nx_n-x_{n+1}=0\}$ lies in $\sigma(T_1, ..., T_n)$, then $\lb_j$ is an eigenvalue of $T_j$ for each $j$ \cite[Lemma 1.14]{Y1}. Thus, the joint spectrum (\ref{thm4.1}) indicates that the orthogonal projections ${\mathscr P}_{j,{S}}$ are all equal, and their range is a common invariant subspace of $A_1,...,A_n$.
\end{proof}

Evidently, the common invariant subspace mentioned in Theorem \ref{self adjoint} is the range of the projections ${\mathscr P}_{i,{S}}$, which are identical for $1\leq i\leq n$. 
Given any admissible transformation by $C=(c_{ij})$ (see (\ref{transform})), we let ${\mathscr P}_{i,{S}}(C)$ denote the projections associated with $\widehat{A}_i, 1\leq i\leq n$ (see \ref{projection S self adjoint})). Since admissible transformations preserve common invariant subspaces of $A_1, ..., A_n$, the projections ${\mathscr P}_{i,{S}}(C)$ are constant with respect to $C$. 

Conversely, suppose that for some  $i$ the projections ${\mathscr P}_{i,{S}}(C)$ are constant with respect to $C$ in some open set. Without  loss of generality we may assume that $i=1$. Let $L$ be the range of ${\mathscr P}_{1,{S}}(C)$. Then $L$ is the same for all $C$ in this open set. This subspace is invariant under the action of
$c_{11}A_1+\cdots +c_{1n}A_n$ for all vectors $(c_{11},...,c_{1n})$ corresponding to the first row of $C$. Choose $n$ matrices $C^1,...,C^n$ in the open set where ${\mathscr P}_{1,{S}}(C)$ is constant in such a way that  the matrix
$$\widetilde{C}=\left [ \begin{array}{ccc} c_{11}^1 & \cdots & c_{1n}^1 \\ \cdot & \cdot & \cdot \\ c_{11}^n & \cdots & c_{1n}^n  \end{array}\right ] $$
is invertible and write
$\widetilde{A}=\widetilde{C}(A) $. Since $(A_1,...,A_n)=\widetilde{C}^{-1}\big(\widetilde{A}\big)$, and $L$ is invariant for the tuple $\widetilde{A}$, we see that $L$ is a common invariant subspace of $A_1,...,A_n$.

This gives rise to a criteria for determining whether a component $\cup_{s\in {\mathcal S} }\Gamma_s$ of the projective joint spectrum corresponds to a common invariant subspace of the matrices.

\begin{corollary}\label{corollary constant projection}
Let $A_1,...,A_n$ be a tuple of self-adjoint $N\times N$ matrices. A component $\cup_{s\in {\mathcal S} }\Gamma_s$ corresponds to a common invariant subspace if and only if ${\mathscr P}_{1,S}(C)$ is constant with respect to $C$. \end{corollary}

\begin{example}\label{example 1}
\normalfont
Consider the following pair of matrices \cite{ST}:
$$A_1=\left [ \begin{array}{ccc} 1&0&0\\0&5&0\\0&0&0\end{array} \right ], \,A_2=\left [ \begin{array}{ccc} 1&2&1\\ 2&7&1\\ 1&1&1/2\end{array}\right ]. $$
The projective joint spectrum for this pair is
$$\sigma(A_1,A_2)=\Big\{ \big(x_1+x_2-x_3\big)\big(\frac{5}{2}x_1x_2+\frac{5}{2}x_2^2-5x_1x_3-\frac{15}{2}x_2x_3+x_3^2\big)=0  \Big\},$$
which consists of two irreducible components: 
\begin{align*}
\Gamma_1&=\Big\{  R_1(x) = x_1+x_2-x_3=0\Big\}, \\
\Gamma_2&=\Big\{ R_2(x)=\frac{5}{2}x_1x_2+\frac{15}{2}x_2^2-5x_1x_3-\frac{15}{2}x_2x_3+x_3^2=0 \Big\}.
\end{align*}
But neither component corresponds to any common invariant subspaces. Let us use Corollary \ref{corollary constant projection} to verify this fact. We consider component $\Gamma_2$. Assume an admissible transform $(A_1, A_2)\to (\widehat{A}_1, \widehat{A}_2)$ is given by a real matrix $C$. For simplicity, we write the first row of $C$ as $(c_1,c_2)$. The eigenvalues of $\widehat{A}_1=c_1A_1+c_2A_2$ are
\begin{align*}
\frac{1}{t_{1,1,1}}&=c_1+c_2,  \\
\frac{1}{t_{2,1,i}}&=\frac{(10c_1+15c_2)\pm \sqrt{100c_1^2+260c_1c_2+105c_2^2}}{4},   \ \ i=1,2.
\end{align*}
A direct computation yields that the projection defined in \eqref{projection S self adjoint} is
\begin{eqnarray*}
& {\mathscr P}_{1,1}(C)=\frac{1}{c_2^2 -\frac{3}{2}c_1^2-\frac{11}{2}c_1c_2} \\
&\times \Big(
c_1^2(A_1^2-5A_1)+c_1c_2(A_1A_2+A_2A_1-\frac{15}{2}A_1-5A_2+\frac{5}{2}I) \\
& +c_2^2(A_2^2-\frac{15}{2}A_2+\frac{15}{2}I)\Big).
\end{eqnarray*}
\noindent By Corollary \ref{corollary constant projection}, in order for $\Gamma_1$ to correspond to a common invariant subspace, the derivatives of ${\mathscr P}_{1,1}(C)$ with respect to $c_1$ and $c_2$ should vanish at every point $(c_1,c_2)\in \R^2$ that might be included as a row in $C$. Since the set of admissible transformations is open and dense, these derivatives must vanish identically. However,
\begin{align*}
\frac{\partial {\mathscr P}_{1,1}}{\partial c_2}\Big|_{(1,0)}&=\frac{-\frac{3}{2}(A_1A_2+A_2A_1-\frac{15}{2}A_1-5A_2+\frac{5}{2}I)+\frac{15}{2}(A_1^2-5A_1)}{\frac{9}{4}} \\
&=\left [ \begin{array}{rrr} -8 & -\frac{4}{3} & \frac{8}{3} \\ -\frac{4}{3} & 0 & \frac{8}{3} \\ \frac{8}{3} & \frac{8}{3} & 0\end{array} \right ] \neq 0.
\end{align*}
In light of Corollary \ref{corollary constant projection}, the component $\Gamma_2$ does not correspond to a common invariant subspace of $(A_1, A_2)$. Since $A_1$ and $A_2$ are self-adjoint, their invariant subspaces are reducing. Therefore, the component $\Gamma_1$ also does not correspond to a common invariant subspace. Hence the pair $(A_1,A_2)$ is irreducible. $\ \ \ \ \ \ \Box$
\end{example}

In practice, using Corollary \ref{corollary constant projection} amounts to verifying that certain elements of the free algebra generated by $A_1,...,A_n$ vanish. These elements depend on the coefficients of the characteristic polynomial. We illustrate how this works with the following example. 

\begin{example}\label{example 2}
\normalfont Let $A_1,A_2$ and $A_3$ be self-adjoint $4\times 4$ matrices, whose projective joint spectrum is given by:
\begin{eqnarray}
&\sigma(A_1,A_2,A_3)\nonumber \\
=&\Big\{(x_1^2+x_2^2+x_3^2-x_1x_2-x_1x_3+2x_2x_3-x_4^2) \nonumber \\
&\times ( x_1^2+x_2^2+x_3^2-x_1x_2-x_1x_3-x_2x_3-x_4^2)=0 \Big\}. \label{spectrum affine}
\end{eqnarray}
We want to determine whether the tuple $(A_1,A_2,A_3)$ is reducible. Here we have 2 components:
\begin{align*}
\Gamma_1&=\Big\{R_1(x)	=x_1^2+x_2^2+x_3^2-x_1x_2-x_1x_3+2x_2x_3-x_4^2=0\Big\}, \\
\Gamma_2&=\Big\{R_2(x)=x_1^2+x_2^2+x_3^2-x_1x_2-x_1x_3-x_2x_3-x_4^2=0\Big\}.
\end{align*}

Assume an admissible transform $(A_1, A_2, A_3)\to (\widehat{A}_1, \widehat{A}_2, \widehat{A}_3)$ is given by a real matrix $C$, whose first row is written as $(c_1,c_2, c_3)$. The eigenvalues of $c_1A_1+c_2A_2+c_3A_3$ are given by
\begin{align*}
\frac{1}{t_{1,1,i} }&=\pm \sqrt{c_1^2+c_2^2+c_3^2-c_1c_2-c_1c_3+2c_2c_3}, \ \ \ i=1,2 \\
\frac{1}{t_{1,2,i}}&=\pm \sqrt{c_1^2+c_2^2+c_3^2-c_1c_2-c_1c_3-c_2c_3}, \ \ \ i=1,2 
\end{align*}
The polynomial \eqref{polynomial q self adjoint} in this case is
\begin{eqnarray*}
&q_{1,1}(z)=z^2-c_1^2-c_2^2-c_3^2+c_1c_2+c_1c_3-2c_2c_3.
\end{eqnarray*}
It follows that $q_{1,1}\big(\frac{1}{t_{1,2,i}}\big)=-3c_2c_3, i=1,2, \alpha_{1,1}=9c_2^2c_3^2$, and
\begin{align*} 
{\mathscr P}_{1,1}(C)=\frac{1}{9c_2^2c_3^2}\Big( &\big(c_1A_1+c_2A_2+c_3A_3\big)^2-\big(c_1^2+c_2^2+c_3^2-\\
&c_1c_2-c_1c_3+c_2c_3\big)I\Big)^2.
\end{align*}
Relation \eqref{spectrum affine} yields $A_1^2=A_2^2=A_3^2=I$, so we have
\begin{eqnarray}
&{\mathscr P}_{1,1}(C)= \frac{9}{c_2^2c_3^2}\Big(c_1c_2\big(A_1A_2+A_2A_1+I\big)+c_1c_3(A_1A_3+A_3A_1+I\big) \nonumber \\
&+c_2c_3(A_2A_3+A_3A_2-I)\Big)^2 \label{projection in affine A 3}
 \end{eqnarray}
It follows from \eqref{spectrum affine} that the joint spectrum of the pair $(A_1,A_2)$ is
\begin{equation}\label{d 3}
\sigma(A_1,A_2)= \big\{ [x_1:x_2:x_3]\in \C\Pp^2: \ (x_1^2+x_2^2-x_1x_2- x_3^2)^2=0\big\} .
\end{equation}
Theorems 5.1 and 5.14 in \cite{CST} show that the pair $(A_1,A_2)$ is unitary equivalent to the pair
\begin{equation}\label{2 dimensional}
\left [ \begin{array}{rrrr} 1 & 0 &0 &0\\ 0 & -1 & 0& 0 \\ 0&0& 1&0 \\0&0&0&-1 \end{array}\right ], \ \ \left [ \begin{array}{rrrr} -1/2 & \sqrt{3}/2 & 0&0\\ \sqrt{3}/2 & 1/2 &0&0\\ 0&0&-1/2 & \sqrt{3}/2\\ 0& 0& \sqrt{3}/2 & 1/2 \end{array}\right ], 
\end{equation}
\noindent so it is easily verified that $ A_1A_2+A_2A_1+I=0.$ A similar argument shows that $A_1A_3+A_3A_1+I=0. $ It follows from \eqref{projection in affine A 3}, the projection ${\mathscr P}_{1,1}(C)$ does not depend on $C$, and we conclude by Corollary \ref{corollary constant projection} that the tuple is reducible. \ \ \ \ \ \ \ \ \ \ $\Box$
\end{example}

\vspace{.2cm}

\textbf{Remark}. Recall that the group $\widetilde{A}_2$ is a group generated by 3 elements $g_1,g_2,g_3$ satisfying the relations
$$g_j^2=1,  \ (g_ig_j)^3=1, \ i,j=1,2,3, \ i\neq j. $$
It is known that $\widetilde{A}_2$ is infinite group. It follows from Example \ref{example 2} that a triple $(A_1,A_2,A_3)$ whose projective joint spectrum is given by \eqref{spectrum affine} determines a 4 dimensional representation of $\widetilde{A}_3$, and this representation is determined by its projective joint spectrum uniquely up to the unitary equivalence.

\subsection{Non self-adjoint tuples}
We remark that if the tuple is not self-adjoint, the condition in Theorem \ref{self adjoint} is necessary but not sufficient.

\begin{example}
\normalfont Let
$$A_1=\left[\begin{array}{cc}1&1 \\ 0&0 \end{array} \right ], \ A_2=\left[\begin{array}{cc}0&0 \\ 1&1 \end{array} \right ]. $$
It is easily seen that the pair is admissible and $\sigma(A_1,A_2)=(x_1+x_2-x_3)x_3$. Thus, ${\mathscr P}_1=A_1$ and ${\mathscr P}_2=A_2$, but they have no nontrivial common invariant subspaces. \ \ \ \ $\Box$
\end{example}

To proceed with the discussion, we recall a theorem from \cite{S1}. Let $A_\Lambda=diag(\lambda_1,...,\lambda_N)$ be any diagonal matrix with distinct entries on the main diagonal and let $T$ be the unilateral shift matrix defined by \[ Te_N=0,\ \ \ Te_i=e_{i+1}, \ \ i=1, ..., N-1,\] where $\{e_1, ..., e_N\}$ is the standard orthonormal basis for $\C^N$. Then it is shown that two $N\times N$ matrices $A$ and $B$ are identical if and only if $\sigma(A_\Lambda, T,T^*, A)=\sigma(A_\Lambda, T,T^*,B)$. In other words, the projective joint spectrum of the tuple $(A_\Lambda, T,T^*, A)$ completely determines the matrix $A$. Moreover, this result also holds for $N=\infty$ under the condition that the set $\{\lambda_1, \lambda_2, ...\}$ is bounded and none of $\lambda_j$ is its accumulation point. This fact readily leads to the following sufficient but not necessary condition regarding the existence of common invariant subspaces.

\begin{theorem}\label{equal projections}
If $\sigma(A_\Lambda, T,T^*,{\mathscr P}_{1,{S}},...,{\mathscr P}_{n,{S}}) $ is invariant under the permutation of variables $x_4,...,x_{n+3}$, then the tuple $(A_1,...,A_n)$ has a common invariant subspace corresponding to the component $\cup_{s\in {S} }\Gamma_s$.
\end{theorem}

\begin{proof}
The invariance of the joint spectrum under the permutations shows that for every $j,k$, we have
$$\sigma(A_\lambda, T,T^*,{\mathscr P}_{j, {S}})=\sigma(A_\lambda, T,T^*,{\mathscr P}_{k, {S}}). $$
By the result mentioned above, we obtain ${\mathscr P}_{j, {S}}={\mathscr P}_{k, {S}}$.
\end{proof}

To see that this condition is not necessary for the existence of a common invariant subspace, we consider the following example.

\begin{example}\normalfont

Let $$A_1=\left [ \begin{array}{cccc}1&0&1&0\\ 0&1&0&1\\0&0&0&0\\0&0&0&0\end{array} \right ], \ A_2= \left [ \begin{array}{cccc}1&0&0&1\\ 0&1&1&0\\0&0&0&0\\0&0&0&0\end{array} \right ].$$
One checks that this pair is admissible and ${\mathscr P}_j=A_j$. Even though $\sigma(A_\lambda, T,T^*,{\mathscr P}_{1})\neq\sigma(A_\lambda, T,T^*,{\mathscr P}_2) $, the subspace generated by $e_1,e_2$ is invariant for both matrices. \ \ \ $\Box$
\end{example}

Aiming to obtain a necessary and sufficient condition for the existence of common invariant subspaces of $A_1, ..., A_n$, we consider the tuple $({\mathscr P}_{1, S}{\mathscr P}_{1, S}^*,..., {\mathscr P}_{n, S}{\mathscr P}_{n, S}^*)$, which consists of non-negative matrices with constant rank $M_S:=\sum_{s\in S} m_{s}$. Let $\mu_{1,j},...,\mu_{r_j,j}$ be the positive eigenvalues of ${\mathscr P}_{j,S}{\mathscr P}_{j, S}^*$ and define
\begin{align}
Q_j(z)&:=1-\frac{\prod_{t=1}^{r_j}(\mu_{t,j}-z)}{\prod_{t=1}^{r_j} \mu_{t,j}}, \ z\in \C, \\
{\mathscr Q}_{j,S}&:=Q_j({\mathscr P}_{j,S}{\mathscr P}_{j,S}^*). \label{polynomials Q}
\end{align}

\begin{theorem}\label{main}
The following are equivalent:
\begin{itemize}
\item[(a)] The collection of components $\cup_{s\in {S}} \Gamma_s$ corresponds to a common invariant subspace; 
\item[(b)] $\sigma({\mathscr Q}_{1,S},...,{\mathscr Q}_{n,S})=\big\{ (x_1+\cdots + x_n-x_{n+1})^{M_{S}} x_{n+1}^{N-M_{S}}=0\big\}$.
\item[(c)] $\sigma(A_\Lambda, T,T^*,{\mathscr Q}_{1,S},...,{\mathscr Q}_{n,S}) $ is invariant under the permutation of variables $x_4,...,x_{n+3}$.
\end{itemize}
\end{theorem}

\begin{proof}
(a) $\Longrightarrow$ (b) and (c). The invariant subspace is the range of each ${\mathscr Q}_{j,S}$, so the projections ${\mathscr Q}_{j,S}$ are the same, and (b) and (c) follow.

(b) $\Longrightarrow$ (a). The range of ${\mathscr Q}_{j,S}$ is the same as the range of ${\mathscr P}_{j,S}$. This shows that ${\mathscr Q}_{j,S}$ is the orthogonal projection. Again applying the local spectral analysis we see that, like in the proof of Theorem \ref{self adjoint}
$$ {\mathscr Q}_{j,S} {\mathscr Q}_{k,S} {\mathscr Q}_{j,S}= {\mathscr Q}_{j,S},$$
an, as in that proof, it implies ${\mathscr Q}_{j,S}={\mathscr Q}_{k,S}, \ j,k=1,...,n$. Thus, ranges of all ${\mathscr P}_{j, S} $ are the same, and the range of these projections is a common invariant subspace.

(c) $\Longrightarrow$ (a). As it was mentioned in the proof of Theorem \ref{self adjoint}, the invariance under the permutations of variables $x_4,...,x_{n+3}$ implies  ${\mathscr Q}_{j,S}={\mathscr Q}_{k,S}, \ j,k=1,...,n$.
	\end{proof}
	
\noindent \textbf{Remark}. Theorem \ref{main} shows that $\cup_{ s \in {S}} \Gamma_s$ determines a common invariant subspace iff projective joint spectrum of specific elements in the unital $C^*$ algebra generated by $A_1,...,A_n,A_1^*,...,A_n^*$ given by \eqref{projection S} and \eqref{polynomials Q} satisfies Theorem \ref{main} (b).

For any admissible transform of $(A_1, ..., A_n)$ by matrix $C$, we set  ${\mathscr Q}_{j,S}(C):=Q_j({\mathcal P}_{j, S}(C){\mathcal P}^*_{j, S}(C))$.
Like in the case of self-adjoint tuples, a result similar to Corollary \ref{corollary constant projection} holds. In fact, the following is a direct consequence of Corollary \ref{corollary constant projection}.

\begin{corollary}
Let $A_1,...,A_n$ be tuple of $N\times N$ matrices. The union of the components $\cup_{s\in {\mathcal S} }\Gamma_s \subset \sigma(A_1,...,A_n)$
corresponds to a common invariant subspace if and only if ${\mathscr Q}_{1,S}(C)$ is constant with respect to $C$. 	
\end{corollary}

\end{document}